\documentclass{amsart}
              [1997/02/02 v1.2e PROC Author Class]

\copyrightinfo{2006}
  {American Mathematical Society}

\newtheorem{theorem}{Theorem}[section]

\newtheorem{lemma}[theorem]{Lemma}

\newcommand{\ov}{\overline}

\newcommand{\ve}{\varepsilon}
\newcommand{\one}{[-1,1]}
\newcommand{\tm}{\times}
\newcommand{\stm}{\setminus}
\newcommand{\rn}{\bf {R^{N}}}
\newcommand{\bn}{{\bf N}}
\newcommand{\br}{{\bf R}}
\newcommand{\ssq}{\subseteq}

\newcommand{\ftr}{f:X\times Y\to \bf R}

\begin{document}

\title[Separately continuous functions with given restriction]{Construction of separately continuous functions with given restriction}

\author{V.V.Mykhaylyuk}
\address{Department of Mathematics\\
Chernivtsi National University\\ str. Kotsjubyn'skogo 2,
Chernivtsi, 58012 Ukraine}
\email{vmykhaylyuk@ukr.net}

\subjclass[2000]{Primary 54C08, 54C30, 54C05}


\commby{Ronald A. Fintushel}


\keywords{separately continuous functions, first Baire class function, diagonal of mapping}

\begin{abstract}
It is solved the problem on constructed of separately continuous functions on product of two topological spaces with given restriction. In particular, it is
shown that for every topological space $X$ and first Baire class function $g:X\to \bf R$ there exists separately continuous function $f:X\times X \to \bf R$
such that $f(x,x)=g(x)$ for every $x\in X$.
\end{abstract}

\maketitle
\section{Introduction}

It ia well-known (see [1]) that the diagonals of separately continuous functions of two real variables are exactly the first Baire functions. It is shown in [2]
that for every topological space $X$ with the normal square $X^2$ and $G_{\delta}$-diagonal and every function $g:X\to \bf R$ of the first Baire class there exists a separately continuous function $f:X\times X\to \bf R$, for which $f(x,x)=g(x)$, i.e. every the first Baire class function on the diagonal can be extended to a separately continuous function on all product. Analogous question for functions of $n$ variables was considered in [3].

On other hand, in the investigations of separately continuous functions $f:X\times Y\to \bf R$ defined on the product of topological spaces $X$ and $Y$ the following two topologies naturally arise (see [4]): the separately continuous topology $\sigma$ (the weakest topology with respect to which all functions $f$ are continuous) and the cross-topology $\gamma$ (it consists of all sets $G$ for which all $x$-sections $G^x=\{y\in Y:(x,y)\in G\}$ and $y$-sections
$G_y=\{x\in X:(x,y)\in G\}$ are open in $Y$ and $X$ respectively). Since the diagonal $\Delta=\{(x,x):x\in \bf R \}$ is a closed discrete set in $(\bf R^2,\sigma )$ or in $(\bf R^2,\gamma)$ and not every function defined on $\Delta$ can be extended to a separately continuous function on $\bf R^2$, even for $X=Y=\bf
R$ the topologies $\sigma$ and $\gamma$ are not normal (moreover, $\gamma$ is not regular [4,5]). Besides, every separately continuous function
$f:X\times Y\to \bf R$ is a Baire class function for a wide class of the products $X\times Y$, in particular, if at least one of the multipliers is matrizable [6]. Thus, the following question naturally arises: for which sets $E\subseteq X\times Y$ and $\sigma$-continuous ($\gamma$-continuous) function $g:E\to \bf
R$ of the first Baire class there exists a separately continuous function $f:X\times Y \to \bf R$ for which the restriction $f|_E$ coincides with $g$?

In this paper we generalize an approach proposed in [2] and solve the problem formulated above for sets  $E$ of some type in the product of topological spaces.

\section{Notions and auxiliary statements}

{\it A set $A\subseteq X$ has the extension property in a topological space $X$}, if every continuous function $g:A\to [0,1]$ can be extended to a continuous function $f:X\to [0,1]$. According to Tietze-Uryson theorem [7, p.116], every closed set in a normal space has the extension property.

\begin{lemma} \label{l:2.1} Let sets $X_1$ and $Y_1$ have the extension property in a topological spaces $X$ and $Y$ respectively, $e:X_1\to Y_1$ be a homeomorphism, $E=\{(x,e(x)):x\in X_1\}$ and $g:E\to [-1,1]$ be a continuous function. Then there exist continuous functions $f:X\times Y\to [-1,1]$ and $h:X\times Y\to [-1,1]$, which satisfies the following conditions:

($i$)    $f|_E=g$;

($ii$)   $E\subseteq h^{-1}(0)$;

($iii$)  for every $x',x''\in X$ and $y',y''\in Y$ if $x'=x''$ or $y'=y''$ then $|f(x',y')-f(x'',y'')|=|h(x',y')-h(x'',y'')|$.
\end{lemma}

\begin{proof} Consider the continuous function $\varphi:X_1\to\one$ and $\psi:Y_1\to\one$, which defined by: $\varphi(x)=g(x,e(x))$, $\psi(y)=g(e^{-1}(y), y)$. Since $X_1$ and $Y_1$ have the extension property in $X$ and $Y$ respectively, there exist continuous functions $\tilde \varphi:X\to\one$ and $\tilde \psi:Y\to\one$ such that $\tilde \varphi|_{X_1}=\varphi$ and $\tilde \psi|_{Y_1}=\psi$. Put $f(x,y)=\frac{\tilde \varphi(x)+\tilde \psi(y)}{2}$ and $h(x,y)=\frac{\tilde \varphi(x)-\tilde \psi(y)}{2}$. Clearly that $f$ and $h$ are continuous on $X\tm Y$ and valued in $[-1,1]$. Moreover, for every point $p=(x,y)\in E$ we have $\tilde \varphi(x)=\varphi(x)=g(p)=\psi(y)=\tilde \psi(y)$. Therefore $f|_E=g$ and $h|_E=0$, i.e. the conditions ($i$) and ($ii$) are hold.

Let $x',x''\in X$ and $y\in Y$. Then
$$f(x',y)-f(x'',y)=\frac{\tilde \varphi(x')-\tilde \varphi(x'')}{2}=h(x',y)-h(x'',y).$$
If $x\in X$ and $y',y''\in Y$, then
$$f(x,y')-f(x,y'')=\frac{\tilde \psi(y')-\tilde \psi(y'')}{2}=h(x,y'')-h(x,y')/$$
Thus, the condition ($iii$) is holds and lemma is proved.
\end{proof}

In the case if the set $E$ satisfies a compactness-type condition we will use the following proposition.

\begin{lemma}\label{l:2.2} Let $X$ be a topological space, $E$ be a pseudocompact set in $X$, $(f_n)_{n=1}^\infty$ be a sequence of continuous functions $f_n:X\to\bf R$ which pointwise converges on the set $E$. Then there exists a functionally closed set $F\ssq X$ such that $E\ssq F$ and the sequence $(f_n)_{n=1}^\infty$ pointwise converges on the set $F$.
\end{lemma}

\begin {proof} Consider the diagonal mapping $$f=\mathop{\Delta} \limits _{n\in{ \bf N}} f_n:X\to {\bf R}^{\bf
N}, \,\,\,f(x)=\left(f_n(x)\right)_{n\in \bf N}.$$ Since the set $E$ is pseudocompact and $f$ is continuous, the set $f(E)$ is a pseudocompact set in the metrizable space $\rn$. Therefore $f(E)$ is closed and the set $F=f^{-1}\left(f(E)\right)$ is functionally closed. It remains to verify that the sequence $(f_n)_{n=1}^\infty$ pointwise converges on $F$. Let $x\in F$. Then there exists an $x_1 \in E$ such that $f(x)=f(x_1)$, i.e. $f_n(x)=f_n(x_1)$ for every $n\in \bn$. Since the sequence $\left(f_n(x_1)\right)_{n=1}^\infty$ is convergent, the sequence $\left(f_n(x)\right)_{n=1}^\infty$ is convergent too.
\end{proof}

The following proposition we will use in a final stage of the construction of separately continuous functions with the given restriction.

\begin{lemma}\label{l:2.3} Let $X$ be a topological space, $F$ be a functionally closed set in $X$, $(h_n)_{n=1}^\infty$ be a sequence of continuous functions $h_n: X\to \br$ such that $F\ssq h_n^{-1}(0)$ for every $n\in \bn$ ³ $G=X\stm F$. Then there exists a locally finite partition of the unit $(\varphi_n)_{n=0}^\infty$ on $G$ such that the supports $G_n=\rm {supp} \varphi_n=$$\{x\in G:\varphi_n (x)>0\}$ of functions $\varphi_n$ satisfy the conditions:

\qquad ($a$) $\ov{G_n}\cap F=\O$ for every $n=0,1,2,\dots$;

\qquad ($b$) $G_n \ssq h_n^{-1}\left((-\frac{1}{n},\frac{1}{n})\right)$ for every $n=1,2,\dots$.
\end{lemma}

\begin{proof} Let $h_0: X\to [0,1]$ be a continuous function such that $F=h_0^{-1}(0)$. For every $n\in \bn$ we put
$$A_n=\mathop{\bigcap}\limits_{k=0}^{n}h_k^{-1}\left((-\frac{1}{n},\frac{1}{n})\right),\,\,\,\,B_n=\mathop{\bigcap}\limits_{k=0}^{n}h_k^{-1}\left([-\frac{1}{n},\frac{1}{n}]\right),$$
$G_n=A_n\stm B_{n+2}$ and, moreover, $G_0=G\stm B_2$. Clearly that all sets $G_n$ are functionally open and $G_n\ssq
h_n^{-1}\left((-\frac{1}{n},\frac{1}{n})\right)$ for every $n\in \bf N$, i.e. the condition $(b)$ holds. Note that $$\mathop{\bigcap}\limits_{n=1}^{\infty}A_n=\mathop{\bigcap}\limits_{n=1}^{\infty}B_n=\mathop{\bigcap}\limits_{n=0}^{\infty}h_n^{-1}(0)=F.$$
Since $A_{n+1}\ssq B_{n+1}\ssq A_n$ for every $n\in \bn$,
$$A_n\stm A_{n+1}\ssq A_n\stm B_{n+2} \ssq A_n\stm A_{n+2}=(A_n\stm
A_{n+1})\mathop{\bigcup}\limits(A_{n+1}\stm A_{n+2}).$$ Therefore
$$
\mathop{\bigcup}\limits_{n=1}^{\infty}G_n=\mathop{\bigcup}\limits_{n=1}^{\infty}(A_n\stm B_{n+2})=
\mathop{\bigcup}\limits_{n=1}^{\infty}(A_n\stm A_{n+1})=A_1\stm (\mathop{\bigcap}\limits_{n=1}^{\infty}A_n)=A_1\stm F.
$$
Thus,
$\mathop{\bigcup}\limits_{n=0}^{\infty}G_n=(G\stm B_2)\cup(A_1\stm F)=G$.

We show that the family $(G_n: n=0,1,\dots)$ is locally finite on $G$. Let $x\in G$, i.e. $h_0(x)\ne 0$. We choose $n_0 \in \bn$ such that  $\frac{1}{n_0}<|h_0(x)|$. Then $x\not \in B_{n_0}$ and the set $G\stm B_{n_0}$ is a neighborhood of $x$. On other hand, $G_n\ssq A_n\ssq B_{n_0}$ for every $n\geq n_0$. Therefore $G_n\cap (G\stm B_{n_0})=\O$ for every $n\geq n_0$. Thus, the family $(G_n:n=0,1,\dots)$ is locally finite at the point $x$.

Since the sets $G_n$ are functionally open, there exist continuous functions $\psi_n: X\to [0,1]$ such that $G_n=\psi_n^{-1}\left((0,1]\right)$. The function $\psi: G\to [0,+\infty)$ which defined by $\psi(x)=\displaystyle\sum\limits_{n=0}^{\infty}\psi_n(x)$ is continuous, moreover, $\rm{supp} \psi=G$. For every $x\in G$ and  $n=0,1,\dots$ we put $\varphi_n(x)=\frac{\psi_n(x)}{\psi (x)}$. The functions $\varphi_n$ are continuous and formed a locally finite partition of the unit on $G$, moreover $G_n=\rm{supp} \varphi_n$.

It remains to verify the condition ($a$). Since $G_n\ssq X\stm B_{n+2}\ssq X\stm A_{n+2}$ and the set $X\stm A_{n+2}$ is closed,
$\ov{G_n}\ssq X\stm A_{n+2}$, i.e. $\ov{G_n}\cap A_{n+2}=\O$. Moreover, $F\ssq A_{n+2}$, therefore $\ov{G_n}\cap F=\O$ for every
$n=0,1,\dots$.
\end{proof}

\section{Main results}

\begin{theorem} \label{th:3.1} Let sets $X_1$ and $Y_1$ have the extension property in topological spaces $X$ and $Y$ respectively, $e: X_1 \to Y_1$ be a homeonorphism, $E=\{(x,e(x)): x\in X_1\}$, $g: E\to \br$ be the first Baire class function and at least one of the following conditions: $E$ is pseudocompact,  $E$ is functionally closed in $X\times Y$, $X_1$ is functionally closed in $X$, $Y_1$ is functionally closed in $Y$ holds. Then there exists a separately continuous function $\ftr$ such that $f|_E=g$.\end{theorem}

\begin{proof} We take a sequence of continuous functions $g_n:E\to [-n,n]$ which pointwise converges to the function $g$ and use Lemma \ref{l:2.1}. We obtain a sequence of continuous functions $f_n: X\times Y \to [-n,n]$ and $h_n: X\times Y \to [-n,n]$ which satisfy the following conditions ($i$)-($iii$).

We show that the set $E$ is contained in some functionally closed set $F_1$ on which the sequence $(f_n)_{n=1}^{\infty}$ pointwise converges. If $E$ is functionally closed, then $F_1=E$. It follows from Lemma \ref{l:2.2} the existence of such set $F_1$ for pseudocompact set $E$. It remains to verify this in the case when $X_1$ or $Y_1$ is functionally closed in $X$ or $Y$ respectively. Let $X_1$ is functionally closed in $X$. Now we put 
$$F_1=(X_1\tm Y)\cap \left(\mathop{\bigcap}\limits_{n=1}^{\infty}h_n^{-1}(0)\right).$$ 
It follows from the property ($ii$) that $E$ is contained in a functionally closed set $F_1$. We take a point $(x,y)$ from the set $F_1$. Using the condition ($iii$) of Lemma \ref{l:2.1} we obtain $|f_n(x,y)-f_n(x,e(x))|=|h_n(x,y)-h_n(x,e(x))|=0$. Hence, $f_n(x,y)=f_n(x,e(x))$. Since the sequence $(f_n)_{n=1}^{\infty}$ pointwise converges on $E$, the sequence $\left(f_n(x,y)\right)_{n=1}^{\infty}$ converges, because $(x,e(x))\in E$.

Now we use Lemma \ref{l:2.3} to the functionally closed set 
$$F=F_1\cap \left(\mathop{\bigcap}\limits_{n=1}^{\infty}h_n^{-1}(0)\right)$$ in the space $X\times Y$ and to the sequence of continuous functions $h_n$ and obtain a locally finite partition of the unit $(\varphi_n)_{n=0}^{\infty}$ on $G=(X\times Y)\stm F$, which satisfies the conditions ($a$) and ($b$).

Let $f_0\equiv 0$ on $X\times Y$. We consider the function
$$
f(x,y)=\left\{
\begin{array}{l}
\displaystyle\sum\limits_{n=0}^{\infty}\varphi_n (x,y) f_n (x,y), \quad \mbox{if $(x,y)\in G$},\\
\lim\limits_{n\to\infty} f_n (x,y), \quad \mbox{if $(x,y)\in F$}.\\
\end{array}\right.
$$
Since $(\varphi_n)_{n=0}^{\infty}$ is a locally finite partition of the unit on the set $G$ and all functions $f_n$ are continuous, the function 
$f$ is correctly defined and continuous on the set $G$. Note that $F\ssq F_1$. Therefore the sequence $(f_n)_{n=0}^{\infty}$ pointwise converges on $F$ and the function $f$ correctly defined on $F$. Moreover, since $E\ssq h_n^{-1}(0)$ for every $n$ and $E\ssq F_1$, $E\ssq F$ and $f|_E=\lim\limits_{n\to\infty}
f_n|_E=\lim\limits_{n\to\infty} g_n=g$.

It remains to verify that the function $f$ is separately continuous at points of the set $F$. Let $p_0=(x_0,y_0)\in F$ ³ $\ve >0$. We choose $n_0\in \bn$ such that $\frac{1}{n_0}<\frac{\ve}{2}$ and $|f_n(p_0)-f(p_0)|<\frac{\ve}{2}$ for every $n\geq n_0$. It follows from the condition ($a$) that the set 
$$W=X\times Y\stm\bigl(\mathop{\bigcup}\limits_{n=0}^{n_0}\ov {G_n}\bigr),$$ where $G_n=\rm{supp} \varphi_n$, is an open neighborhood of $p_0$ in $X\times Y$. We take a neighborhood $U$ of $x_0$ in $X$ such that $U\tm \{y_0\}\ssq W$. Let $x\in U$. If $p=(x,y_0)\in F$, then $h_n(p)=0$ for every $n\in \bn$. Then according to the condition ($iii$), we obtain $|f_n(p_0)-f_n(p)|=|h_n(p_0)-h_n(p)|=0$, i.e. $f_n(p_0)=f_n(p)$. Therefore $f(p_0)=f(p)$. If $p\not\in F$, then $$f(p)=\displaystyle\sum\limits_{n=0}^{\infty}\varphi_n(p)f_n(p)=\displaystyle\sum\limits_{n=n_0}^{\infty}\varphi_n(p)f_n(p),$$
because $p\in W$. Then
$$
|f(p_0)-f(p)|=\bigl|\sum\limits_{n=n_0}^{\infty}\varphi_n(p)\bigl(f(p_0)-f_n(p_0)\bigr)+
\sum\limits_{n=n_0}^{\infty}\varphi_n(p)f_n(p_0)-
$$
$$
-\sum\limits_{n=n_0}^{\infty}\varphi_n(p)f_n(p)\bigr|\leq \sum\limits_{n=n_0}^{\infty}\varphi_n(p)|f(p_0)-
f_n(p_0)|+
$$
$$
+\sum\limits_{n=n_0}^{\infty}\varphi_n(p)|f_n(p_0)-f_n(p)|<\sum\limits_{n=n_0}^{\infty}\varphi_n(p)
\cdot\frac{\ve}{2}+
$$
$$
+\sum\limits_{n=n_0}^{\infty}\varphi_n(p)|h_n(p_0)-h_n(p)|=\frac{\ve}{2}+
\sum\limits_{n=n_0}^{\infty}\varphi_n(p)|h_n(p)|.
$$
It follows from the property ($b$) of sets $G_n$ that if $\varphi_n(p)\ne 0$, then $|h_n(p)|<\frac{1}{n}$. Thus, 
$$
\sum\limits_{n=n_0}^{\infty}\varphi_n(p)|h_n(p)|\leq \sum\limits_{n=n_0}^{\infty}\varphi_n(p)\cdot
\frac{1}{n}\leq\frac{1}{n_0}\sum\limits_{n=n_0}^{\infty}\varphi_n(p)=\frac{1}{n_0}
<\frac{\ve}{2}.
$$
Hence, $|f(p_0)-f(p)|<\ve$. Thus, $f$ is continuous at $p_0$ with respect to $x$.

The continuity of $f$ at $p_0$ with respect to $y$ can be proved analogously. Thus, $f$ is separately continuous and the theorem is proved.
\end{proof}

In the case of $X=Y=X_1=Y_1$ we obtain the following theorem which generalized the result from [2].

\begin{theorem}\label{th:3.2} Let $X$ be a topological space and $g: X\to \br$ be a function of the first Baire class. Then there exists a separately continuous function $f: X\tm X\to\br$ such that $f(x,x)=g(x)$ for every $x\in X$.\end{theorem}

\section{Functions on the product of compacts}

Now we consider the case when $X$ and $Y$ satisfy compactness type conditions. A set $E$ in a product $X\times Y$ is called {\it horizontally and vertically onepointed} if for every $x\in X$ and $y\in Y$ the sets $E\cap (\{x\}\tm Y)$ and $E\cap (X\tm \{y\})$ are at most countable and {\it horizontally and vertically $n$-pointed} if corresponding sets contain at most $n$ elements.

\begin{theorem} \label{th:4.1} Let $X$ and $Y$ be compacts, $E$ be a closed horizontally and vertically onepointed set in $X\times Y$ and $g: E\to\br$ be a function of the first Baire class. Then there exists a separately continuous function $\ftr$, for which $f|_E=g$.\end{theorem}

\begin{proof} Since the set $E$ is horizontally and vertically onepointed, the projections the compact set $E$ to the axis $X$ and $Y$ are continuous injective mappings. Hence,  $E$ is the graph of a homeomorphism $e: X_1\to Y_1$, where $X_1$ and $Y_1$ are the projections of $E$ on $X$ and $Y$ respectively. Now the existence of desired function $f$ follows from Theorem \ref{th:3.1}.
\end{proof}

\begin{theorem} \label{th:4.2} Let $X$ and $Y$ be a locally compact spaces such that $X\times Y$ be a paracompact, $E$ be a closed horizontally and vertically onepointed set and $g: E\to\br$ be the first Baire class function. Then there exists a separately continuous function $\ftr$ for which $f|_E=g$.\end{theorem}

\begin{proof} For every $p=(x,y)\in X\times Y$ we choose open neighborhoods $U_p$ and $V_p$ of $x$ and $y$ in $X$ and $Y$ respectively such that the closure $X_p=\ov {U_p}$ and $Y_p=\ov {V_p}$ are compacts and the set $E_p=E\cap (X_p\tm Y_p)$ is horizontally and vertically onepointed. According to Theorem \ref{th:4.2} there exists a separately continuous function $f_p: X_p\tm Y_p\to\br$ for which $f_p|_{E_p}=g|_{E_p}$. Since the space $X\times Y$ is a paracompact, there exists a partition of the unit $(\varphi_i: i\in I)$ on $X\times Y$ which is subordinated to the open cover $(W_p=U_p\tm V_p: p\in X\times Y)$ of $X\times Y$ [7, p.447]. For every $i \in I$ we choose $p_i\in X\times Y$ such that $\rm{supp}\varphi_i$$\ssq W_{p_i}$ and put
$$
g_i(x,y)=\left\{
\begin{array}{l}
f_{p_i}(x,y), \quad \mbox{if $(x,y)\in W_{p_i}$},\\
0, \quad \mbox{if $(x,y)\not\in W_{p_i}$}.
\end{array}
\right.
$$

Note that the functions $\varphi_i\cdot g_i$ are separately continuous on $X\times Y$ and $(\varphi_i g_i)|_E=(\varphi_i |_E)g$. Then the function
$f=\displaystyle\sum\limits_{i\in I}\varphi_i g_i$ is the required.
\end{proof}

\section{Example}

Finally we give a example which show the essentiality of conditions under the set $E$ in Theorem \ref{th:4.1}.

Let $X=Y=[0,1]$, 
$$E_1=\{(\frac{2k-1}{2^n}, \frac{2k-1}{2^n}+\frac{1}{2^{n+1}}):
k=1,\dots ,2^{n-1}, n\in\bn\},$$
$E_2=\{(x,x): x\in X\}$, $E=E_1\cup E_2$, $g: E\to\br$,
$
g(x)=\left\{
\begin{array}{l}
1, \quad x\in E_1,\\
0, \quad x\in E_2.\\
\end{array}
\right. $ 
Clearly that $E$ is a closed horizontally and vertically $2$-pointed set in $X\times Y$ and $g$ is a function of the first Baire class. Since the set $E_1$ is dense in $E$, the set $D(g)$ of discontinuity points set of the function $g$ coincides with $E$. Therefore the projections of $D(g)$ on the axis $X$ and $Y$ coincide with $X$ and $Y$ respectively. On other hand, it is well-known that for every separately continuous function $\ftr$ the set $D(f)$ of points of discontinuity of $f$ is contained in the product $A\tm B$ of meagre sets  $A\ssq X$ and $B\ssq Y$ respectively. Thus, $D(g)\not\ssq D(f)$ and the function $f$ can not be extension of the function $g$.

\bibliographystyle{amsplain}

\begin{thebibliography}{10}

\bibitem {Ba} Baire R. {\it Sur les fonctions de variable reelles} An. Mat.Pura Appl., ser. 3 (1899), 1-123.

\bibitem {MS} Mykhaylyuk V.V., Sobchuk O.V. {\it Functions with diagonal of finite Baire class} Mat.Studii. $\bf 14$,N1 (2000), 23-28 (in Ukrainian).

\bibitem {MMS} Maslyuchenko V.K., Mykhaylyuk V.V., Sobchuk O.V.{\it Construction of a separately continuous function of $n$ variables with given diagonal} Mat.Studii. $\bf 12$,N1 (1999), 101-107 (in Ukrainian).

\bibitem {HW} Henriksen M., Woods R.G. {\it Separate versus joint continuity: A tale of four topologies} Top. Appl. $\bf 97$,N1-2 (1999) 175-205.

\bibitem {M} Mykhaylyuk V.V. {\it Separately continuous topology and a generalization of a Sierpinski"s theorem} Mat.Studii. $\bf 14$,N2 (2000), 193-196 (in Ukrainian).

\bibitem {Ru} Rudin W. {\it Lebesgue first theorem} Math. Analysis and Applications, Part B. Edited by Nachbin. in Math. Supplem. Studies 78. - Academic
Press (1981), 741-747.

\bibitem {E} Engelking R. {\it General topology} M. Mir, 1986 (in Russian).

\end{thebibliography}

\end{document}